\documentclass[11pt]{article}
\usepackage{bbm}
\usepackage{amssymb}
\usepackage{amsfonts}
\usepackage{mathrsfs}
\usepackage{mathrsfs}
\usepackage{mathrsfs}
\usepackage{amssymb,amsmath,amsthm,amsfonts}
\usepackage{indentfirst}
\usepackage{color}
\usepackage[all]{xy}
\usepackage[numbers,sort&compress]{natbib}

\def\be{\begin{equation}}
\def\ee{\end{equation}}

\usepackage{graphicx}
\newtheorem{theorem}{Theorem}[section]                   

\newtheorem{lemma}{Lemma}[section]

\newtheorem{remark}{Remark}[section]

\newcommand{\N}{\mathbb N}
\newcommand{\Z}{\mathbb Z}
\newcommand{\R}{\mathbb R}

\newcommand{\C}{\mathbb C}

\begin{document}

\title{\LARGE \bf Minimal $P$-symmetric period problem of first-order autonomous   Hamiltonian Systems }
\author{ Chungen Liu and Ben-Xing Zhou\\\
{\small School of Mathematics and LPMC, Nankai University,}\\
{\small Tianjin 300071, P. R. China}\\
{\small }}
\date{}
 \maketitle
 \footnotetext[1]{ Partially supported by the NSF of China (11471170, 10621101), 973 Program of
MOST (2011CB808002) and SRFDP.
$\;\;\;\;$ Email: {\tt liucg@nankai.edu.cn}(for Liu) and {\tt 1120130022@mail.nankai.edu.cn}(for Zhou)}

\noindent {\bf Abstract}: Let $P\in Sp(2n)$ satisfying $P^{k}=I_{2n}$, we consider the minimal $P$-symmetric period problem of the autonomous nonlinear Hamiltonian system
   \begin{equation*}
\dot x(t) = JH^{\prime}(x(t)).
\end{equation*}
For some symplectic matrices $P$, we show that for any $\tau>0$ the above Hamiltonian system possesses a $k\tau$ periodic solution $x$ with $k\tau$ being its minimal $P$-symmetric  period  provided $H$ satisfies the   Rabinowitz's conditions on the minimal period conjecture, together with that $H$ is convex and $H(Px)=H(x)$.

\vspace{0.3cm} \noindent {\bf Key Words:} \ Maslov $P$-index, Relative Morse index, Minimal $P$-symmetric period, Hamiltonian system

\section{Introduction and main result}
\setcounter{equation}{0}

In this paper, we study the following first-order autonomous Hamiltonian system with $P$-boundary condition:
\begin{equation}\label{1.1}
\left\{ \begin{array}{ll}\dot{x}= JH^{\prime}(x), x\in \R^{2n}\\
x(\tau)=Px(0).\;\;\end{array}\right.
\end{equation}
where $\tau>0$, $P\in Sp(2n)$, and $H\in C^{2}(\R^{2n}, \R)$ is the Hamiltonian function satisfying $H(Px) = H(x)$, $\forall x \in \R^{2n}$. $H^{\prime}(x)$ denote its gradient,
$J=\left( \begin{array}{cc}
0 \ \ & -I_{n}\\
I_{n} & 0
\end{array} \right)$
is the standard symplectic matrix, $I_{n}$ is the identity matrix on $\R^{n}$. Without confusion, we shall omit the
subindex of the identity matrix.

A solution $(\tau, x)$ of the problem (\ref{1.1}) is called a $P$-solution of the Hamiltonian systems. The problem (\ref{1.1}) has relation with the closed geodesics on Riemannian manifold (cf.\cite{HS1}) and symmetric periodic solution or the quasi-periodic solution problem (cf.\cite{HS2}). In addition, the first author C. Liu in \cite{Liu2} transformed some periodic boundary problem for nonlinear delay differential systems and some nonlinear delay Hamiltonian systems to $P$-boundary problems of Hamiltonian systems as above, we also refer \cite{CAMR,FT,HW, HS3} and references therein for the background of $P$-boundary problems in $N$-body problems.

Let $P\in Sp(2n)$ and $k\in \N=\{0,1,2,\cdots \}$, we say $P$ satisfies $(P)_{k}$ condition, if $P^{k}=I_{2n}$ and for each integer $m$ with $1\leq m\leq k-1$, $P^{m}\neq I$. If $P$ satisfies $(P)_{k}$ condition, a $P$-solution $(\tau, x)$ can be extended as a $k\tau$-periodic solution $(k\tau, x^{k})$. We say that a $T$-periodic solution $(T, x)$ of the Hamiltonian system in (\ref{1.1}) is $P$-symmetric if $x(\frac{T}{k})=Px(0)$. $T$ is the $P$-symmetric period of $x$. $T$ is called the minimal $P$-symmetric period of $x$ if $T=\min\{\lambda>0\mid x(t+\frac{\lambda}{k})=Px(t), \forall t\in \R\}$.

We assume the following conditions on $H$ in our arguments:
\begin{enumerate}

\item[(H0)] $H\in C^{1}(\R^{2n}, \R)$  $\forall x \in \R^{2n}$;
\item[(H1)] $H\in C^{2}(\R^{2n}, \R)$ with $H(Px) = H(x)$, $\forall x \in \R^{2n}$;

\smallskip
\item[(H2)] $H(x)\geq 0$, $\forall x\in \R^{2n}$;
\item[(H3)] $H(x)= o(\vert x\vert^{2})$ as $\vert x\vert\to 0$;
\smallskip
\item[(H4)]  There are constants $\mu > 2$ and $R_{0} > 0$ such that
\begin{equation*}
0 < \mu H(x) \leq  (H^{\prime}(x), x), \ \ \forall \ \vert x \vert \geq R_{0};
\end{equation*}
\item[(H5)]  $H^{\prime\prime}(x)> 0$, $\forall x\in \R^{2n}$;
\end{enumerate}

In \cite{Rab1}, Rabinowitz proved that the Hamiltonian system in (\ref{1.1}) possesses a non-constant prescribed period solution provided $H$ satisfying (H0) and (H2)-(H4). Because a $\tau/k$-periodic function is also a $\tau$-periodic function, moreover, in \cite{Rab1} Rabinowitz proposed a conjecture: {\it under the conditions (H0) and (H2)-(H4), for any $\tau>0$, the Hamiltonian system in (\ref{1.1}) possesses a $\tau$-periodic with $\tau$ being its minimal period}. Since then, there were many papers on this minimal period problem (cf. \cite{CE}, \cite{AM}, \cite{AC}, \cite{EH}, \cite{Long4}, \cite{Long5}, \cite{Long6}, \cite{DDL}, etc.). In 1997, D. Dong and Y. Long \cite{DDL} developed a new method on this prescribed minimal period solution problem and discovered the intrinsic relationship between the minimal period and the indices of a solution. Based upon the work of \cite{DDL}, G. Fei, Q. Qiu, T. Wang and others  applied this method to various problems of Rabinowitz's conjecture (cf. \cite{FQ}, \cite{FKW}, \cite{Liu4}, etc.). In fact, under conditions $H(0)=0$ and $H(x)> 0$, $\forall x\in \R^{2n} \setminus \{0\}$; $\frac{H(x)}{\mid x \mid^{2}}\rightarrow +\infty$ as $x\rightarrow 0$ and $\frac{H(x)}{\mid x \mid^{2}}\rightarrow 0$ as $x\rightarrow +\infty$, F. Clarke and I. Ekeland proved a result on the corresponding minimal period problem for some given $T$ in \cite{CE}; I. Ekeland and H. Hofer gave a criterion for the conjecture in \cite{EH} which is unfortunately not easy to check.

For $P$-boundary problem, S. Tang \cite{LT2}  and the first author of this paper proved that for any $0<\tau<\frac{\pi}{\max_{t\in [0, \tau]}\Vert J\dot{\gamma}_{P}(t)\gamma_{P}(t)^{-1} \Vert}$, there exists a nonconstant $P$-solution with its minimal $P$-symmetric period  $k\tau$ or $\frac{k\tau}{k+1}$  via the iteration theory of Maslov $P$-index. In \cite{Liu3}, the first author of this paper  improved the result that for every $\tau>0$,
there exists a nonconstant $P$-solution with its minimal $P$-symmetric period  $k\tau$ or $\frac{k\tau}{k+1}$.

For $n\in \N$, $k > 0$, denote by
\begin{equation*}
Sp(2n)\equiv Sp(2n,\R)= \{M \in \mathcal L(\R^{2n}) \mid M^{T}JM=J \},
\end{equation*}
$$\mathcal P_{\tau}(2n)\equiv \{\gamma \in C([0,\tau], Sp(2n)) \mid \gamma(0)= I\},$$
\begin{equation*}
Sp(2n)_{k}\equiv \{P \in Sp(2n) \mid  {   P   \ satisfies\  (P)_{k}\  condition }\},
\end{equation*}
\begin{equation*}
\Omega(M)\equiv \{N \in Sp(2n) \mid  {\rm \sigma(N)\cap \mathbf{U}=\sigma(M)\cap \mathbf{U}\ and\ \nu_{\lambda}(N)=\nu_{\lambda}(M),\forall \ \lambda\in \sigma(M)\cap \mathbf{U} }\}.
\end{equation*}
Denote by $\Omega^{0}(M)$ the path connected component of $\Omega(M)$ which contains $M$.

\begin{lemma}\label{ll}
If $P \in Sp(2n)_{k}$, then there exists a matrix $I_{2p} \diamond R(\frac{2\pi}{k})^{\diamond j_{1}} \diamond \cdots
\diamond R(\frac{2r\pi}{k})^{\diamond j_{r}} \in \Omega^{0}(P^{-1})$, with $p+\sum_{m=1}^{r} j_{m}=n$.
\end{lemma}

\begin{proof}
For $P \in Sp(2n)_{k}$, we have $\sigma(P^{-1})=\sigma(P)\subseteq\{1, e^{\frac{2\pi \sqrt{-1}}{k}}, e^\frac{4\pi \sqrt{-1}}{k}, \cdots, e^\frac{2(k-1)\pi \sqrt{-1}}{k}\}\subseteq \mathbf{U}$. By the Theorem 1.8.10 in \cite{Long1}, there exists $M_{1}(\omega_{1}) \diamond M_{2}(\omega_{2})\diamond \cdots \diamond  M_{s}(\omega_{s}) \in \Omega^{0}(P^{-1})$ where $M_{i}(\omega_{i})$ is a basic normal form of some eigenvalue  of $P^{-1}$, $1\leq i \leq s$. And the following are the basic normal forms for eigenvalues in $\mathbf{U}$.

\noindent {\it Case 1}. $N_{1}(\lambda,b) = \left( \begin{array}{cc}
\lambda  & b\\
0 & \lambda
\end{array} \right)$, $\lambda = \pm 1, b=\pm 1,0$.\\
Since $P \in Sp(2n)_{k}$, we have $b=0$ and $\lambda \in \{-1, 1\} \cap \sigma(P^{-1})$.

\noindent {\it Case 2}. $R(\theta)= \left( \begin{array}{cc}
\cos\theta  & -\sin \theta\\
\sin \theta & \cos \theta
\end{array} \right)$, $\theta \in (0, \pi) \cup (\pi, 2\pi)$.\\
Since $P \in Sp(2n)_{k}$, we have $\theta \in \{\frac{2\pi}{k}, \frac{4\pi}{k}, \cdots \frac{2(k-1)\pi}{k}\}$.

\noindent {\it Case 3}. $N_{2}(\omega,b) = \left( \begin{array}{cc}
R(\theta)  & b\\
0 & R(\theta)
\end{array} \right)$, $\theta \in (0, \pi) \cup (\pi, 2\pi)$, $b=\left( \begin{array}{cc}
b_{1}  & b_{2}\\
b_{3} & b_{4}
\end{array} \right), b_{i}\in \R, b_{2} \neq b_{3}$.

From direct computation, it is easy to check that the matrix
$T = \left(
 \begin{array}{cc}
 \frac{1}{\sqrt{2}} & \frac{\sqrt{-1} }{\sqrt{2}} \\
 \frac{\sqrt{-1}}{\sqrt{2}} & \frac{1}{\sqrt{2}} \\
 \end{array}
 \right)$
satisfies that $T R(\theta) T^{-1} =
\left(
 \begin{array}{cc}
  e^{\sqrt{-1}\theta} & 0 \\
  0 & e^{-\sqrt{-1}\theta} \\
  \end{array}
  \right)$.
Then
$\left(
\begin{array}{cc}
T & 0 \\
0 & T \\
\end{array}
\right)
N_{2}(\omega,b)
\left(
\begin{array}{cc}
T^{-1} & 0 \\
0 & T^{-1} \\
\end{array}
\right) =
\left(
  \begin{array}{cc}
    T R(\theta) T^{-1} & T b T^{-1} \\
    0 & T R(\theta) T^{-1} \\
  \end{array}
\right)$,
where
\begin{equation}
T b T^{-1} =
\left(
\begin{array}{cc}
\frac{1}{2}(b_{1}+b_{4})-\frac{\sqrt{-1}}{2}(b_{2}-b_{3}) & \frac{1}{2}(b_{2}+b_{3})-\frac{\sqrt{-1}}{2}(b_{1}-b_{4}) \\
\frac{1}{2}(b_{2}+b_{3})+\frac{\sqrt{-1}}{2}(b_{1}-b_{4}) & \frac{1}{2}(b_{1}+b_{4})+\frac{\sqrt{-1}}{2}(b_{2}-b_{3}) \\
\end{array}
\right).
\end{equation}

Denoted by
\begin{equation}
\left(
  \begin{array}{cc}
   T R(i\theta) T^{-1} & X(i) \\
    0 & T R(k\theta) T^{-1} \\
  \end{array}
\right) =
\left(
  \begin{array}{cc}
    T R(\theta) T^{-1} & T b T^{-1} \\
    0 & T R(\theta) T^{-1} \\
  \end{array}
\right)^{i}, \;i\in \mathbb{N}
\end{equation}
where
$X(i) =
\left(
  \begin{array}{cc}
    x_{1}(i) & x_{2}(i) \\
    x_{3}(i) & x_{4}(i) \\
  \end{array}
\right)$
and $X(1) = T b T^{-1} =
\left(
  \begin{array}{cc}
    x_{1}(1) & x_{2}(1) \\
    x_{3}(1) & x_{4}(1) \\
  \end{array}
\right)$.
By direct computation, we have
\begin{equation}
\begin{split}
x_{1}(k)&=ke^{\sqrt{-1}(k-1)\theta}x_{1}(1),\\
x_{4}(k)&=ke^{-\sqrt{-1}(k-1)\theta}x_{4}(1).
\end{split}
\end{equation}
Thus, from $P^{k}=I$ we have $X(k)=0$, so $x_{1}(1)=x_{4}(1)=0$, i.e.
\begin{equation*}
\begin{split}
\frac{1}{2}(b_{1}+b_{4})-\frac{\sqrt{-1}}{2}(b_{2}-b_{3})=0,\\
\frac{1}{2}(b_{1}+b_{4})+\frac{\sqrt{-1}}{2}(b_{2}-b_{3})=0.
\end{split}
\end{equation*}
Then we have $b_{2}=b_{3}$, which is contradict to the definition of the basic normal form $N_{2}(\omega,b)$.

Therefore, from (Case1)-(Case3),we get $M_{i}(\omega_{i})=R(\theta_{i})$ where $\theta_{i}\in \{0, \frac{2\pi}{k}, \frac{4\pi}{k}, \cdots \frac{2(k-1)\pi}{k}\}$, $1\leq i \leq s$. And the lemma is proved.
\end{proof}

For the notations in Lemma \ref{ll}, we define
\begin{equation*}
\begin{split}
Sp(2n)_{k}(r,p;j_{1},j_{2},\cdots,j_{r}) \equiv \left\{P \in Sp(2n)_{k} \mid k-2\sum_{m=1}^{r}m\cdot j_{m} >1 ,r < \frac{k}{2}\right\}.
\end{split}
\end{equation*}

Now we state the main result of this paper.

\begin{theorem}\label{theorem:1}
Suppose $P \in Sp(2n)_{k}(r,p;j_{1},j_{2},\cdots,j_{r})$, and the Hamiltonian function $H$ satisfies (H1)-(H5), then for every $\tau > 0$,
the system (\ref{1.1}) possesses a non-constant {\it $P$-solution} $(\tau, x)$ such that the minimal $P$-symmetric period of the extended $k\tau$-periodic solution $(k\tau, x^{k})$ is $k\tau$.
\end{theorem}

In order to prove the above result, we need to obtain the relationship between the Maslov $P$-index and Morse index. Thus we organize this paper as follows, in Section 2, we recall the definition and properties of the Maslov $P$-index theory, and we also list out the relationship between the  Morse index and the Maslov $P$-index (see \cite{LT1},\cite{LT2}, \cite{Liu3}, \cite{Fei} and \cite{FQ}). In Section 3, we first study the iteration formila of Maslov index of paths $\xi \in \mathcal P_{\tau}(2n)$ such that $\xi(\tau) = P^{-1}$ in detail, then we will give the complete proof of the main result.

\section{Preliminaries}
\setcounter{equation}{0}
In this section, we give a brief introduction to the Maslov $P$-index and its iteration properties, and then give the relationship between Maslov $P$-index and the relative Morse index which is studied by the first author of this paper in \cite{Liu3}.

Maslov $P$-index was first studied in \cite{Dong} and \cite{Liu1} independently for any symplectic matrix $P$ with different treatment. The first author and S. Tang in \cite{LT1,LT2} defined the Maslov $(P,\omega)$-index $(i_{\omega}^{P}(\gamma), \nu_{\omega}^{P}(\gamma))$ for any symplectic path $\gamma \in \mathcal P_{\tau}(2n)$. And then the first author of this paper used relative index theory to develop Maslov $P$-index in \cite{Liu3} which is consistent with the definition in \cite{LT1,LT2}. When the symplectic matrix $P=diag \{-I_{n-\kappa}, I_{\kappa}, -I_{n-\kappa}, I_{\kappa} \}$, $0\leq \kappa \in \N \leq n$, the $(P,\omega)$-index theory and its iteration theory were studied in \cite{DL} and then be successfully used to study the multiplicity of closed characteristics on partially symmetric convex compact hypersurfaces in $\R^{2n}$. Here we use the notions and results in \cite{Liu1,LT1,LT2}.

For $\omega \in \mathbf{U}$, then the Maslov $(P,\omega)$-index of a symplectic path $\gamma \in \mathcal P_{\tau}(2n)$ is defined as a pair of integers(cf.\cite{LT1})
$$(i_{\omega}^{P}(\gamma), \nu_{\omega}^{P}(\gamma))\in \Z\times \{0,1,\cdots,2n\},$$
where the index part
\begin{equation}
i_{\omega}^{P}(\gamma)= i_{\omega}(P^{-1}\gamma \ast \xi) - i_{\omega}(\xi),
\end{equation}
$\xi \in \mathcal P_{\tau}(2n)$ such that $\xi(\tau)=P^{-1}$ and the nullity
\begin{equation}
 \nu_{\omega}^{P}(\gamma)=\dim \ker(\gamma(\tau)-\omega P).
 \end{equation}
Suppose $B(t)\in C(\R, \mathcal{L}_{s}(\R ^{2n}))$, if $\gamma \in \mathcal P_{\tau}(2n)$ is the fundamental solution of the linear Hamiltonian systems
\begin{equation}\label{2}
\dot y(t)= JB(t)y,\ \ \ y\in \R^{2n},
\end{equation}
we also call $(i_{\omega}^{P}(\gamma), \nu_{\omega}^{P}(\gamma))$ the Maslov $(P,\omega)$-index of $B(t)$, denoting
$(i_{\omega}^{P}(B), \nu_{\omega}^{P}(B))=(i_{\omega}^{P}(\gamma), \nu_{\omega}^{P}(\gamma))$, just as in \cite{Liu1,LT1,LT2}.
 If $x$ is a $P$-solution of (\ref{1.1}), then the Maslov $(P,\omega)$-index of the solution $x$ is defined to be the Maslov $(P,\omega)$-index of $B(t)=H^{\prime\prime}(x(t))$ and denoted by $(i_{\omega}^{P}(x), \nu_{\omega}^{P}(x))$. When $\omega = 1$, we omit the subindex, denoted by $(i^{P}(\gamma), \nu^{P}(\gamma))$ or $(i^{P}(B), \nu^{P}(B))$ for simplicity.

For $m\in \N$, we extend the definition of $x(t)$ which is the solution of (1.1) to $[0, +\infty)$ by
\begin{equation*}
x(t) = P^{j}x(t-j\tau),\ \ \forall j\tau \leq t \leq (j+1)\tau,\ j\in \N
\end{equation*}
and define the $m$-th iteration $x^{m}$ of $x$ by
\begin{equation*}
x^{m} = x \vert_{[0, m\tau]}.
\end{equation*}
If $P$ satisfies $(P)_{k}$ condition, then $x^{k}$ becomes an $k\tau$-periodic solution of the Hamiltonian system in (1.1). We know that the fundamental solution $\gamma_{x} \in \mathcal P_{\tau}(2n)$ carries significant information about $x$. For any $\gamma \in \mathcal P_{\tau}(2n)$,  S. Tang and  the first author  of this paper have defined the corresponding $m$-th iteration path $\gamma^{m} : [0, m\tau] \to Sp(2n)$ of $\gamma$ in \cite{LT1} by
\begin{equation}\label{4}
\gamma^{m}(t) =
\begin{cases}
\gamma(t), & t \in [0, \tau],\\
P\gamma(t-\tau)P^{-1}\gamma(\tau), & t \in [\tau, 2\tau],\\
P^{2}\gamma(t-2\tau)(P^{-1}\gamma(\tau))^{2}, & t \in [2\tau, 3\tau],\\
P^{3}\gamma(t-3\tau)(P^{-1}\gamma(\tau))^{3}, & t \in [3\tau, 4\tau],\\
\cdots \cdots \\
P^{m-1}\gamma(t-(m-1)\tau)(P^{-1}\gamma(\tau))^{m-1}, & t \in [(m-1)\tau, m\tau].\\
\end{cases}
\end{equation}
If the matrix function $B(t)$ in the linear Hamiltonian system (\ref{2}) satisfies $P^{T}B(t+\tau)P = B(t)$, the iteration of its fundamental solution $\gamma$ is defined in the same way.

Corresponding we set
\begin{equation*}
i_{\omega}^{P^{m}}(\gamma, m) = i_{\omega}^{P^{m}}(\gamma^{m}),\ \ \nu_{\omega}^{P^{m}}(\gamma, m) = \nu_{\omega}^{P^{m}}(\gamma^{m}).
\end{equation*}

If the subindex $\omega = 1$, we simply write $(i^{P^{m}}(\gamma, m), \nu^{P^{m}}(\gamma, m))$, and omit the subindex $1$ when there is no confusion. In the sequel, we use the notions $(i(\gamma), \nu(\gamma))$  and $(i(\gamma, m), \nu(\gamma, m))$ to denote the Maslov-type index and the iterated index of symplectic path $\gamma$ with the periodic boundary condition which were introduced by Y. Long and his collaborators (cf.\cite{LL1}, \cite{LL2}, \cite{Long1}, \cite{LZ2}, etc.).\\

In \cite{LT1,LT2},  S. Tang and the first author of this paper obtained the important Bott-type formula and iteration inequalities for Malsov $(P,\omega)$-index as follows.

\begin{lemma}(\cite{LT1}, Bott-type iteration formula)\label{lemma:1}
For any $\tau > 0$, $\gamma \in \mathcal P_{\tau}(2n)$ and $m \in \N$, there hold
\begin{equation}
i^{P^{m}}_{\omega_0}(\gamma, m) = \sum_{\omega^{m}=\omega_0}i^{P}_{\omega}(\gamma),    \ \ \nu^{P^{m}}_{\omega_0}(\gamma, m) = \sum_{\omega^{m}=\omega_0}\nu^{P}_{\omega}(\gamma),
\end{equation}
\end{lemma}
\noindent where $P\in Sp(2n)$ and $\omega_{0}\in \bf U$.

\begin{lemma}(\cite{LT2})\label{lemma:2}
 For any path $\gamma \in \mathcal P_{\tau}(2n)$, $P \in Sp(2n)$ and $\omega \in \mathbf{U} \setminus \{1\}$, it always holds that
\begin{equation}\label{12}
i^{P}(\gamma, 1) + \nu^{P}(\gamma, 1) - n + i_{1}(\xi) - i_{\omega}(\xi) \leq i_{\omega}^{P}(\gamma) \leq i^{P}(\gamma, 1) + n - \nu_{\omega}^{P}(\gamma) + i_{1}(\xi) - i_{\omega}(\xi).
\end{equation}
\end{lemma}

\begin{lemma}(\cite{LT2}, iteration inequality )\label{lemma:3}
For any path $\gamma \in \mathcal P_{\tau}(2n)$, $P \in Sp(2n)$ and $m \in \mathbf{N}$,
\begin{equation}\label{14}
  \begin{split}
  &m(i^{P}(\gamma, 1) + \nu^{P}(\gamma, 1) - n) + n -\nu^{P}(\gamma, 1) + mi_{1}(\xi)-i(\xi, m)\\
  &\leq i^{P^{m}}(\gamma, m)\\
  &\leq m(i^{P}(\gamma, 1) + n) - n - (\nu^{P^{m}}(\gamma, m) - \nu^{P}(\gamma, 1)) + mi_{1}(\xi)-i(\xi, m).
  \end{split}
  \end{equation}
\end{lemma}
Let $e(M)$ be the elliptic height of symplectic matrix $M$ just as the same in \cite{Long1}, the following lemma is important for the proof of Theorem \ref{theorem:1}.
\begin{lemma}[\cite{LT2}]\label{lemma:4}
For any path $\gamma \in \mathcal P_{\tau}(2n)$, $P \in Sp(2n)$, set $M = \gamma(\tau)$ and extend $\gamma$ to $[0, \infty)$ by (\ref{4}). Then for any $m \in \mathrm{N}$ we have
\begin{equation}
\begin{aligned}
&\nu^{P^{m}}(\gamma, m) - \nu(\xi, 1) + \nu(\xi, m+1) - \frac{e(P^{-1}M)}{2} - \frac{e(P^{-1})}{2}\\
&\leq i^{P^{(m+1)}}(\gamma, m+1) - i^{P^{m}}(\gamma, m) - i^{P}(\gamma, 1)\\
&\leq \nu^{P}(\gamma, 1) - \nu^{P^{(m+1)}}(\gamma, m+1) - \nu(\xi, m) + \frac{e(P^{-1}M)}{2} + \frac{e(P^{-1})}{2}.
\end{aligned}
\end{equation}
\end{lemma}

Let $S_{k\tau}=\R/(k\tau\Z)$, $W_{P}=\{z\in W^{1/2,2}(S_{k\tau}, \R^{2n})\mid z(t+\tau)=Pz(t)\}$ be a closed subspace of $W^{1/2,2}(S_{k\tau}, \R^{2n})$. It is also a Hilbert space with norm $\Vert\cdot\Vert$ and inner product $\langle\cdot, \cdot\rangle$ as in $W^{1/2,2}(S_{k\tau}, \R^{2n})$. We denote by $\Vert \cdot \Vert_{s}$ the $L^{s}$-norm for $s\geq 1$. By the well-known Sobolev embedding theorem, we have the following embedding property: for any $s\in [1, +\infty)$, there is a constant $\alpha_{s}>0$ such that
\begin{equation}\label{34}
\Vert z \Vert_{s} \leq \alpha_{s}\Vert z \Vert,\ \ \forall z\in W_{P}.
\end{equation}
 Let $\mathcal{L}_{s}(W_{P})$ and  $\mathcal{L}_{c}(W_{P})$ denote the space of the bounded self-adjoint linear operator and compact linear operator on $W_{P}$. We define two operators $A$, $B\in \mathcal{L}_{s}(W_{P})$ by the following bilinear forms:
\begin{equation}\label{3}
\langle Ax, y \rangle = \int_{0}^{\tau}(-J\dot{x}(t), y(t))dt,\ \ \langle Bx, y \rangle = \int_{0}^{\tau}(B(t)x(t), y(t))dt.
\end{equation}

Suppose that $\cdots \leq \lambda_{-j}\leq \cdots \leq \lambda_{-1}<0<\lambda_{1}\leq \cdots \leq \lambda_{j}\leq \cdots$ are all nonzero eigenvalues of $A$ (count with multiplicity), correspondingly, $e_{j}$ is the eigenvector of $\lambda_{j}$ satisfying $\langle e_{j}, e_{i}\rangle=\delta_{ji}$.
We denote the kernel of $A$ by $W_{P}^{0}$ which is exactly the space $\ker_{\R}(P-I)$.
For $m \in \N$, define the finite dimensional subspace of $W_{P}$ by
\begin{equation*}
W^{m}_{P} = W_{m}^{-}\oplus W_{P}^{0}\oplus W_{m}^{+}
\end{equation*}
with $W_{m}^{-}=\{z\in W_{P}\vert z(t)=\sum_{j=1}^{m}a_{-j}e_{-j}(t), a_{-j}\in\R\}$ and $W_{m}^{+}=\{z\in W_{P}\vert z(t)=\sum_{j=1}^{m}a_{j}e_{j}(t),\\
a_{j}\in\R\}$. Suppose $P_{m}$ is the orthogonal projections $P_{m}: W_{P} \to W_{P}^{m}$ for $m\in\N\cup \{0\}$. Then $\{P_{m} \mid m=0, 1, 2, \cdots\}$ is the Galerkin approximation sequence respect to $A$.

For a self-adjoint operator $T$, we denote by $M^{\ast}(T)$ the eigenspaces of $T$ with eigenvalues belonging to $(0,+\infty),\{0\}$ and
$(-\infty,0)$ with $\ast = +,0$ and $\ast = -$, respectively.
And the dimension of eigenspaces $M^{\ast}(T)$ is denoted by $m^{\ast}(T)= \dim M^{\ast}(T)$. Similarly, we denote by $M_{d}^{\ast}(T)$ the eigenspaces of $T$ with eigenvalues belonging to $(d,+\infty),(-d, d)$ and
$(-\infty,-d)$ with $\ast = +,0$ and $\ast = -$, respectively. we denote $m_{d}^{\ast}(T)= \dim M_{d}^{\ast}(T)$.
For any adjoint operator $L$, we denote $L^{\sharp}= (L|_{Im L})^{-1}$.

The following theorem gives the relationship between the Maslov $P$-index and  Morse index for any $P \in Sp(2n)$.

\begin{theorem}(\cite{Liu3}, Lemma 3.2 and Theorem 4.6)\label{theorem:2}
For $P\in Sp(2n)$, suppose that $B(t)\in \C(\R, \mathcal{L}_{s}(\R^{2n}))$ and $P^{T}B(t+\tau)P = B(t)$ with the Maslov $P$-index $(i^{P}(B), \nu^{P}(B))$. For any constant $0<d\leq\frac{1}{4}\Vert (A - B)^{\sharp} \Vert^{-1}$, there exists an $m_{0}>0$ such that for $m\geq m_{0}$, there holds
\begin{equation}\label{35}
\begin{split}
m_{d}^{-}(P_{m}(A - B)P_{m})&=m+i^{P}(B),\\
m_{d}^{0}(P_{m}(A - B)P_{m})&=\nu^{P}(B),\\
m_{d}^{+}(P_{m}(A - B)P_{m})&=m + \dim\ker _{\R}(P-I)-i^{P}(B)-\nu^{P}(B),
\end{split}
\end{equation}
where $B$ is the operator defined by $B(t)$.
\end{theorem}

For the operators $A$ and $B$ defined in (\ref{3}), there is another description of the Maslov $P$-index as follows.

\begin{lemma}[\cite{Liu3}]\label{lemma:5}
For any two operators $B_{1}, B_{2} \in C(\R, \mathcal{L}_{s}(2n))$ with $ B_{i}(t+\tau)=(P^{-1})^{T}B_{i}(t)P^{-1}, i=1,2$ and $B_{1} < B_{2}$, there holds
\begin{equation}
i^{P}(B_{2}) - i^{P}(B_{1}) = \sum_{s\in[ 0, 1 )}^{}\nu^{P}((1-s)B_{1} + sB_{2}).
\end{equation}
\end{lemma}

\begin{remark}\label{remark:1}
Suppose that $B > 0$, we have
\begin{equation}
i^{P}( B ) = \sum_{s\in[ 0, 1 )}^{}\nu^{P}( sB ).
\end{equation}
\end{remark}

\section{The proof of Theorem \ref{theorem:1}}
\setcounter{equation}{0}

In \cite{Liu3}, the following result was proved.

\begin{theorem}(\cite{Liu3})\label{theorem:3}
Suppose $P \in Sp(2n)_{k}$, and the Hamiltonian function $H$ satisfies (H1)-(H4),
then for every $\tau > 0$, the system (\ref{1.1}) possesses a nonconstant {\it $P$-solution} $(\tau, x)$ satisfying
\begin{equation}
\rm{\text{dim}\ \text{ker}_{\R}}(P-I)+2- \nu^{P}(\it x ) \leq i^{P}(\it x ) \leq \rm{\text{dim}\ \text{ker}_{\R}}(P-I)+1.
\end{equation}

\end{theorem}

Before the proof of Theorem \ref{theorem:1}, we need to get the information about the iteration Maslov index for paths connecting $I$ and $P^{-1}$. Firstly, from Lemma \ref{ll} we recall that for any $P\in Sp(2n)_{k}$, there exists $(p,j_{1},j_{2},\cdots,j_{r})\in \N^{r+1}$ such that $p+\sum_{m=1}^{r} j_{m}=n$ and $I_{2p} \diamond R(\frac{2\pi}{k})^{\diamond j_{1}} \diamond \cdots
\diamond R(\frac{2r\pi}{k})^{\diamond j_{r}} \in \Omega^{0}(P^{-1})$. We remind that here $\N=\{0,1,\cdots,\}$.
\begin{lemma}\label{lemma:6}
For $P\in Sp(2n)_{k}(r,p;j_{1},j_{2},\cdots,j_{r})$ and $\xi \in \mathcal{P}_{\tau}(2n)$ with $\xi(\tau)=P^{-1}$, there holds
\begin{equation}
(k+1)i(\xi)-i(\xi,k+1)=\sum_{m=1}^{r}(k-2m) j_{m} -kp.
\end{equation}
\end{lemma}

\begin{proof}
By Theorem 9.3.1 in \cite{Long1}(also \cite{LZ2}), we have
\begin{equation}
\begin{aligned}
i(\xi,k+1)&=(k+1)(i(\xi)+S^{+}_{P^{-1}}(1)-C(P^{-1}))\\&\quad+2\sum_{\theta\in (0,2\pi)}E(\frac{(k+1)\theta}{2\pi})S^{-}_{P^{-1}}(e^{\sqrt{-1}\theta})-(S^{+}_{P^{-1}}(1)+C(P^{-1}))\\
&=(k+1)i(\xi)+kS^{+}_{P^{-1}}(1)-(k+2)C(P^{-1})+2\sum_{\theta\in (0,2\pi)}E(\frac{(k+1)\theta}{2\pi})S^{-}_{P^{-1}}(e^{\sqrt{-1}\theta}),
\end{aligned}
\end{equation}
where $S^{\pm}_{M}(\omega)$ denote the splitting number of $M\in Sp(2n)$ at $\omega\in \mathbf{U}$, $C(M)=\sum_{\theta\in (0,2\pi)}S^{-}_{M}(e^{\sqrt{-1}\theta})$ and $E(a)=\min\{m\in \Z\mid m\geq a\}$. One can see these notions in Chapter 9 of \cite{Long1}.

For $P\in Sp(2n)_{k}(r,p;j_{1},j_{2},\cdots,j_{r})$, $I_{2p} \diamond R(\frac{2\pi}{k})^{\diamond j_{1}} \diamond \cdots
\diamond R(\frac{2r\pi}{k})^{\diamond j_{r}} \in \Omega^{0}(P^{-1})$ with $2r<k$. By direct computation, for $\theta\in (0,\pi)\cup (\pi,2\pi)$, we have
\begin{equation}
\begin{split}
S^{+}_{P^{-1}}(1)&=p\ ;\\
C(P^{-1})&=\sum_{m=1}^{r}j_{m}\ ;
\end{split}
\end{equation}

\begin{equation}
S^{-}_{P^{-1}}(e^{\sqrt{-1}\theta})=
\begin{cases}
j_{m}, & if \ \theta=\frac{2m\pi}{k}, 1\le m\le r,\; e^{\frac{2m\pi\sqrt{-1}}{k}}\in \sigma(P^{-1})\ ;\\
\ 0, & otherwise.
\end{cases}
\end{equation}

So
\begin{equation}
\begin{aligned}
\sum_{\theta\in (0,2\pi)}E(\frac{(k+1)\theta}{2\pi})S^{-}_{P^{-1}}(e^{\sqrt{-1}\theta})&=E(\frac{k+1}{k})j_{1}+E(\frac{2(k+1)}{k})j_{2}+\cdots+E(\frac{r(k+1)}{k})j_{r}\\
&=2j_{1}+3j_{2}+\cdots+(r+1)j_{r}.
\end{aligned}
\end{equation}
Then
\begin{equation}
\begin{aligned}
i(\xi,k+1)&=(k+1)i(\xi)+kp-(k+2)(j_{1}+j_{2}+\cdots+j_{r})+2(2j_{1}+3j_{2}+\cdots+(r+1)j_{r})\\
&=(k+1)i(\xi)+kp-(k-2)j_{1}-(k-4)j_{2}-\cdots-(k-2r)j_{r}\ ,
\end{aligned}
\end{equation}
thus $(k+1)i(\xi)-i(\xi,k+1)=\sum_{m=1}^{r}(k-2m) j_{m} -kp$.
\end{proof}

\bigskip
Now we are ready to prove Theorem \ref{theorem:1}.

\noindent {\bf Proof of Theorem \ref{theorem:1}.}

Suppose that $(k\tau, x^{k})$ is the $k\tau$-periodic solution extended by $P$-solution $(\tau, x)$ in Theorem \ref{theorem:3}.
If $k\tau$ is not the minimal {\it $P$-symmetric} period of $(k\tau, x^{k})$, i.e., $\tau>\min\{\lambda>0\mid x(t+\lambda)=Px(t), \forall t\in \R\}$, then there exists some $l\in \N$ such that
\begin{equation*}
T\equiv\frac{\tau}{l}= \min\{\lambda>0\mid x(t+\lambda)=Px(t), \forall t\in \R\}.
\end{equation*}
Thus $x(\tau-T)=x(0)$, both $(l-1)T$ and $kT$ are the period of $x$. Since $kT$ is the minimal {\it $P$-symmetric} period, we obtain $kT\leq (l-1)T$ and then $k\leq l-1$.

Note that $x\vert_{[0, kT]}$ is the $k$-th iteration of $x\vert_{[0, T]}$.
Suppose $\gamma\in \mathcal{P}_{T}(2n)$ is the fundamental solution of the following linear Hamiltonian system
\begin{equation}\label{lh}
\dot{z}(t) = JB(t)z(t)
\end{equation}
with $B(t)=H^{\prime\prime}(x\vert_{[0, T]}(t))$.
Suppose $\xi$ be any symplectic path in $\mathcal P_{T}(2n)$ such that $\xi(T) = P^{-1}$, since $P^{k}=I$, then
\begin{equation}\label{61}
\nu(\xi, 1)=\nu(\xi, k+1)=\nu(\xi, l).
\end{equation}
All eigenvalues of $P$ and $P^{-1}$ are on the unit circle, then the elliptic height
\begin{equation}\label{68}
e(P^{-1})=e(P)=2n.
\end{equation}
Since the system (\ref{1.1}) is autonomous, we have
\begin{equation}\label{62}
\nu_{1}(x\vert_{[0, kT]})\geq 1 \ \ \text{and}\ \ \nu^{P^{l-1}}(\gamma, l-1)=\nu_{1}(x\vert_{[0, (l-1)T]})\geq 1.
\end{equation}
By Lemma \ref{lemma:4}, $P^{l-1}=I$ and (\ref{61})-(\ref{62}), we have
\begin{equation}\label{63}
\begin{split}
i^{I}(\gamma, l-1)&=i^{P^{l-1}}(\gamma, l-1)\\
&\leq i^{P^{l}}(\gamma, l)-i^{P}(\gamma, 1)+\nu(\xi, 1)-\nu(\xi, l)+\frac{e(P^{-1}\gamma(T))}{2} + \frac{e(P^{-1})}{2}-\nu^{P^{l-1}}(\gamma, l-1)\\
&\leq i^{P^{l}}(\gamma, l)-i^{P}(\gamma, 1)+\frac{e(P^{-1}\gamma(T))}{2}+n-1\\
&\leq i^{P^{l}}(\gamma, l)-i^{P}(\gamma, 1)+2n-1.
\end{split}
\end{equation}
Note that $i^{P^{l}}(\gamma, l)=i^{P}(x\mid_{[0, \tau]})\leq \text{dim\ ker}_{\R}(P-I_{2n}) + 1$, here we write $i^{P}(x\mid_{[0, \tau]})$ for $i^{P} (x)$ to remind the solution $x$ is defined in the interval $[0, \tau]$. By the definition of Maslov $P$-index,
\begin{equation*}
i^{I}(\gamma, l-1)=i_{1}(\gamma, l-1)+n.
\end{equation*}
So we get
\begin{equation}\label{64}
i_{1}(\gamma, l-1) \leq \dim\ker_{\R}(P-I) - i^{P}(\gamma, 1) + n.
\end{equation}

By the condition (H5) and Remark \ref{remark:1}, we have
\begin{equation}
i^{P}(\gamma,1)=i^{P}(B)= \sum_{s\in [0,1)}^{}\nu^{P}(sB)=\sum_{s\in [0,1)}^{} \dim\ker_{\R}(\gamma_{B}(sT)-P).
\end{equation}
Here we remind that $B(t)$ and $\gamma_{B}$ are defined in (\ref{lh}). Since $\gamma_{B}(0)=I$, so $\dim\ker_{\R}(\gamma_{B}(sT)-P)=\dim\ker_{\R}(P-I)$ when $s=0$. Thus we have
\begin{equation}
i^{P}(\gamma,1)\geq \dim\ker_{\R}(P-I).
\end{equation}
From (3.12), it implies
\begin{equation}\label{65}
i_{1}(\gamma, l-1)\leq n.
\end{equation}
 By the convex condition (H5), we also have
\begin{equation}\label{66}
i_{1}(x\vert_{[0, kT]})\geq n \ \ \text{and}\ \ i_{1}(x\vert_{[0, (l-1)T]})\geq n.
\end{equation}
We set $m=\frac{l-1}{k}$. Note that $x\vert_{[0, (l-1)T]}$ is the $m$-th iteration of $x\vert_{[0, kT]}$. By (\ref{62}), (\ref{65}), (\ref{66}) and Lemma 4.1 in \cite{LL2}, we obtain $m=1$ and then $k=l-1$. From the above process (3.12)-(3.13) and (3.15)-(3.17), we obtain $k=l-1$ provided $e(P^{-1}\gamma(T))=2n$, and
\begin{equation}\label{67}
\begin{split}
i^{P}(\gamma,1)&= \dim\ker_{\R}(P-I);\\
i^{P}(\gamma,l)&=i^{P}(\gamma,k+1)= \dim\ker_{\R}(P-I)+1,\\
\nu^{P^{k}}(\gamma, k)&=\nu^{P}(\gamma, 1)=1.
\end{split}
\end{equation}
Here we remind that the left inequality in (\ref{14}) of Lemma \ref{lemma:3} holds independent of the choice of $\xi\in\mathcal{P}_{\tau}(2n)$, then for any $\xi\in\mathcal{P}_{\tau}(2n)$ we have
\begin{equation}\label{68}
i^{P}(\gamma,k+1)\geq (k+1)(i^{P}(\gamma, 1) + \nu^{P}(\gamma, 1) - n) + n - 1 + (k+1)i_{1}(\xi)-i(\xi, k+1).
\end{equation}
By the condition $P\in Sp(2n)_{k}(r,p;j_{1},j_{2},\cdots,j_{r})$, we get
\begin{equation}
\text{dim\ ker}_{\R}(P-I) = 2p,
\end{equation}
\begin{equation}
k-2\sum_{m=1}^{r}m\cdot j_{m} > 1.
\end{equation}
Applying (3.18), (3.20) and Lemma \ref{lemma:6} to (3.19), we get
\begin{equation}
k-2\sum_{m=1}^{r}m\cdot j_{m} \leq 1.
\end{equation}
It is contradict to the inequality (3.21). So the minimal {\it $P$-symmetric} period of $(k\tau, x^{k})$ is $k\tau$.

\begin{remark}
Note that $e(P^{-1}\gamma(T))=e((P^{-1}\gamma(T))^{l})=e(P^{-1}\gamma(\tau))=2n$ is required in the above proof.
If $e(P^{-1}\gamma(T))\leq 2n-2$, we get $i_{1}(\gamma, l-1)<n$ by taking the same process as (\ref{63})-(\ref{64}). It contradicts to the second inequality of (\ref{66}). At this moment, the minimal {\it $P$-symmetric} period of $(k\tau, x^{k})$ is $k\tau$.

The condition (H5) can be replaced by a weaker condition: $H''(x(t))\ge 0$ and $\int_0^{\tau}H''(x(t))dt>0$ for the $P$-solutuin $(\tau,x)$ in Theorem 3.1.
\end{remark}

\end{document}